\documentclass[12pt]{mathscan}
\usepackage{amsfonts,amssymb,amscd,amsmath,enumerate,color, xypic}
\def\NZQ{\mathbb}               

\def\ZZ{{\NZQ Z}}
\def\RR{{\NZQ R}}

\def\frk{\mathfrak}               

\def\Phi{{\frk N}}
\def\eb{{\bold e}}

\def\opn#1#2{\def#1{\operatorname{#2}}} 
\opn\chara{char} 
\opn\length{\ell} 
\opn\pd{pd} 
\opn\rk{rk}
\opn\projdim{proj\,dim} 
\opn\injdim{inj\,dim} 
\opn\rank{rank}
\opn\depth{depth} 
\opn\grade{grade} 
\opn\height{height}
\opn\embdim{emb\,dim} 
\opn\codim{codim}

\opn\Tr{Tr} 
\opn\bigrank{big\,rank}
\opn\superheight{superheight}
\opn\lcm{lcm}
\opn\trdeg{tr\,deg}
\opn\reg{reg} 
\opn\lreg{lreg} 
\opn\ini{in} 
\opn\lpd{lpd}
\opn\size{size}
\opn\mult{mult}
\opn\dist{dist}
\opn\cone{cone}
\opn\lex{lex}
\opn\rev{rev}
\opn\div{div} \opn\Div{Div} \opn\cl{cl} \opn\Cl{Cl}
\opn\Spec{Spec} \opn\Supp{Supp} \opn\supp{supp} \opn\Sing{Sing}
\opn\Ass{Ass} \opn\Min{Min}
\opn\Ann{Ann} \opn\Rad{Rad} \opn\Soc{Soc}
\opn\Syz{Syz} \opn\Im{Im} \opn\Ker{Ker} \opn\Coker{Coker}
\opn\Am{Am} \opn\Hom{Hom} \opn\Tor{Tor} \opn\Ext{Ext}
\opn\End{End} \opn\Aut{Aut} \opn\id{id} \opn\ini{in}

\opn\nat{nat}
\opn\pff{pf}
\opn\Pf{Pf} \opn\GL{GL} \opn\SL{SL} \opn\mod{mod} \opn\ord{ord}
\opn\Gin{Gin}
\opn\Hilb{Hilb}\opn\adeg{adeg}\opn\std{std}\opn\ip{infpt}
\opn\Pol{Pol}
\opn\sat{sat}
\opn\Var{Var}
\opn\Gen{Gen}

\opn\aff{aff} \opn\con{conv} \opn\relint{relint} \opn\st{st}
\opn\lk{lk} \opn\cn{cn} \opn\core{core} \opn\vol{vol}
\opn\link{link} \opn\star{star}
\opn\gr{gr}

\def\Ac{{\mathcal A}}

\def\Jc{{\mathcal J}}
\def\Gc{{\mathcal G}}

\def\Oc{{\mathcal O}}
\def\Pc{{\mathcal P}}

\def\Cc{{\mathcal C}}
\def\pot#1#2{#1[\kern-0.28ex[#2]\kern-0.28ex]}

\opn\dirlim{\underrightarrow{\lim}}
\opn\inivlim{\underleftarrow{\lim}}
\let\to=\rightarrow

\def\Implies{\ifmmode\Longrightarrow \else
	\unskip${}\Longrightarrow{}$\ignorespaces\fi}
\def\implies{\ifmmode\Rightarrow \else
	\unskip${}\Rightarrow{}$\ignorespaces\fi}
\def\iff{\ifmmode\Longleftrightarrow \else
	\unskip${}\Longleftrightarrow{}$\ignorespaces\fi}

\let\:=\colon
\newtheorem{Theorem}{Theorem}[section]

\newtheorem{Corollary}[Theorem]{Corollary}

\newtheorem{Remark}[Theorem]{Remark}

\newtheorem{Example}[Theorem]{Example}

\let\epsilon\varepsilon
\let\phi=\varphi
\let\kappa=\varkappa
\def\qed{\ifhmode\textqed\fi
	\ifmmode\ifinner\quad\qedsymbol\else\dispqed\fi\fi}
\def\textqed{\unskip\nobreak\penalty50
	\hskip2em\hbox{}\nobreak\hfil\qedsymbol
	\parfillskip=0pt \finalhyphendemerits=0}
\def\dispqed{\rlap{\qquad\qedsymbol}}

\opn\dis{dis}
\opn\height{height}
\opn\dist{dist}
\def\pnt{{\raise0.5mm\hbox{\large\bf.}}}

\opn\Lex{Lex}
\opn\conv{conv}

\begin{document}
\title{Quadratic Gr\"{o}bner bases arising from\\ partially ordered sets}
\author[T. Hibi]{Takayuki Hibi}
\address[Takayuki Hibi]{Department of Pure and Applied Mathematics,
	Graduate School of Information Science and Technology,
	Osaka University,
	Toyonaka, Osaka 560-0043, Japan}
\email{hibi@math.sci.osaka-u.ac.jp}
\author[K. Matsuda]{Kazunori Matsuda}
\address[Kazunori Matsuda]{Department of Pure and Applied Mathematics,
	Graduate School of Information Science and Technology,
	Osaka University,
	Toyonaka, Osaka 560-0043, Japan}
\email{kaz-matsuda@math.sci.osaka-u.ac.jp}
\author[A. Tsuchiya]{Akiyoshi Tsuchiya}
\address[Akiyoshi Tsuchiya]{Department of Pure and Applied Mathematics,
	Graduate School of Information Science and Technology,
	Osaka University,
	Toyonaka, Osaka 560-0043, Japan}
\email{a-tsuchiya@cr.math.sci.osaka-u.ac.jp}

\begin{abstract}
The order polytope $\mathcal{O}(P)$ and the chain polytope $\mathcal{C}(P)$ associated to 
a partially ordered set $P$ are studied. 
In this paper, we introduce the convex polytope $\Gamma(\mathcal{O}(P), -\mathcal{C}(Q))$ which is the convex hull of 
$\mathcal{O}(P) \cup (-\mathcal{C}(Q))$, where both $P$ and $Q$ are partially ordered sets
with $|P|=|Q|=d$. 
It will be shown that $\Gamma(\mathcal{O}(P), -\mathcal{C}(Q))$ is a normal and Gorenstein Fano polytope 
by using the theory of reverse lexicographic squarefree initial ideals of toric ideals. 

\end{abstract} 

\maketitle

\section*{introduction}

A convex polytope $\Pc \subset \RR^{d}$ is {\em integral} if all vertices belong to $\ZZ^{d}$. 
An integral convex polytope $\Pc \subset \RR^{d}$ is {\em normal} if, for each integer $N > 0$ 
and for each ${\bf a} \in N \Pc \cap \ZZ^{d}$, there exist ${\bf a}_1, \ldots, {\bf a}_N \in \Pc \cap \ZZ^{d}$ 
such that ${\bf a} = {\bf a}_1 + \cdots +{\bf a}_N$, where $N \Pc = \{ N \alpha \mid \alpha \in \Pc \}$. 
Furthermore, an integral convex polytope $\Pc \subset \RR^{d}$ is {\em Fano} if the origin of $\RR^{d}$ 
is a unique integer point belonging to the interior of $\Pc$. 
A Fano polytope $\Pc \subset \RR^{d}$ is {\em Gorenstein} 
if its dual polytope 
\[
\Pc^{\vee} := \{ {\bf x} \in \RR^{d} \mid \langle {\bf x}, {\bf y} \rangle \le 1 \  \mathrm{for\ all}  \ {\bf y} \in \Pc \}
\]
is integral as well. 
A Gorenstein Fano polytope is also said to be a reflexive polytope. 

In recent years, the study of Gorenstein Fano polytopes has been more vigorous. 
It is known that Gorenstein Fano polytopes correspond to Gorenstein Fano varieties, and 
they are related with mirror symmetry (see, e.g., \cite{mirror}, \cite{Cox}). 
On the other hand, to find new classes of Gorenstein Fano polytopes is one of the most important problem. 

As a way to construct normal Gorenstein Fano polytopes, taking the {\em centrally symmetric configuration} \cite{CSC} of an integer matrix is a powerful tool. 
In \cite{CSC}, it is shown that,  for any matrix $A$ with $\rank (A) = d$ such that all nonzero maximal minors 
of $A$ are $\pm 1$, the integral convex polytope arising from the centrally symmetric configuration of $A$ 
is normal Gorenstein Fano. 
Moreover, in \cite{harmony}, a way to construct non-symmetric normal Gorenstein Fano polytopes 
is introduced.  
In this paper, we treat the integral convex polytopes arising from combining the order polytopes and 
chain polytopes associated with two partially ordered sets. 

Let $P=\{p_1,\ldots,p_d\}$ and $Q=\{q_1,\ldots,q_d\}$ be finite partially ordered sets (posets, for short) 
with $|P|=|Q|=d$.
A subset $I$ of $P$ is called a {\em poset ideal} of $P$ if $p_{i} \in I$ and $p_{j} \in P$ together with $p_{j} \leq p_{i}$ guarantee $p_{j} \in I$.  
Note that the empty set $\emptyset$ as well as $P$ itself is a poset ideal of $P$. 
Let $\Jc(P)$ denote the set of poset ideals of $P$.
A subset $A$ of $Q$ is called an {\em antichain} of $Q$ if
$q_{i}$ and $q_{j}$ belonging to $A$ with $i \neq j$ are incomparable.  
In particular, the empty set $\emptyset$ and each 1-element subsets $\{q_j\}$ are antichains of $Q$.
Let $\Ac(Q)$ denote the set of antichains of $Q$.

Let $\eb_{1}, \ldots, \eb_{d}$ the canonical unit coordinate vectors of $\RR^{d}$.  
Then, for each subset $I \subset P$ and for each subset $J$ of $Q$, 
we define the $(0, 1)$-vectors $\rho(I) = \sum_{p_{i}\in I} \eb_{i}$ and $\rho(J) = \sum_{q_{j}\in J} \eb_{j}$, 
respectively. 
In particular $\rho(\emptyset)$ is the origin ${\bf 0}$ of $\RR^{d}$. 
Recall that the {\em order polytope} $\Oc(P)$ (\cite[Definition 1.1]{Stanley}) 
is the convex hull of $\{ \rho(I) \mid I \in \Jc(P) \}$ 
and the {\em chain polytope} $\Cc(Q)$ (\cite[Definition 2.1]{Stanley}) 
is the convex hull of $\{ \rho(A) \mid A \in \Jc(Q) \}$. 

Now, we define the convex polytope 
$\Gamma(\Oc(P), - \Cc(Q))$ as the convex hull of $\Oc(P) \cup -(\Cc(Q))$, 
where $- \Cc(Q) = \{ - \beta \mid \beta \in \Cc(Q) \}$.
This is a kind of $(-1, 0, 1)$-polytope, that is, 
each of its vertices belongs to $\{-1, 0, 1\}^{d}$. 
Note that $\dim \Gamma(\Oc(P), - \Cc(Q)) = d$.  
Moreover, since $\rho(P) = \eb_{1} + \cdots + \eb_{d} \in \Oc(P)$ and 
$\rho(\{q_j\})=\eb_{j} \in \Cc(Q)$ for $1 \leq j \leq d$, 
we have that the origin ${\bf 0}$ of $\RR^{d}$
belongs to the interior of $\Gamma(\Oc(P), - \Cc(Q))$.

$\Gamma(\Oc(P), - \Cc(Q))$ is an integral convex polytope arising from
combining the order polytope $\Oc(P)$ and the chain polytope $\Cc(Q)$ 
associated with two posets $P$ and $Q$.  
We consider the question when such polytope is normal Gorenstein Fano.  
This kind of question has been studied. 
It is known that $\Gamma(\Oc(P), - \Oc(P))$ is normal Gorenstein Fano for any poset $P$ \cite{HMOS} and 
$\Gamma(\Oc(P), - \Oc(Q))$ is normal Gorenstein Fano 
if $P$ and $Q$ possess a common linear extension \cite{twin}. 
Moreover, it is shown that $\Gamma(\Cc(P), - \Cc(Q))$ is normal Gorenstein Fano for any posets $P$ and $Q$ \cite{harmony}. 

In this paper, we prove that $\Gamma(\Oc(P), - \Cc(Q))$ is normal Gorenstein Fano 
for any posets $P$ and $Q$ 
by using the theory of reverse lexicographic squarefree initial ideals of toric ideals.
For fundamental materials on Gr\"{o}bner bases and toric ideals, see \cite{dojoEN}. 


 \section{Squarefree Quadratic Gr\"{o}bner basis}
Let, as before, $P = \{ p_{1}, \ldots, p_{d} \}$ and $Q = \{ q_{1}, \ldots, q_{d} \}$ 
be finite partially ordered sets with the same cardinarity. 
For a poset ideal $I \in \Jc(Q)$, we write $\max(I)$ for the set of maximal elements of $I$.
In particular, $\max(I)$ is an antichain of $Q$.  
Note that for each antichain $A$ of $Q$, there exists a poset ideal $I$ of $Q$ such that 
$A=\max(I)$. 

Let $K[{\bf t}, {\bf t}^{-1}, s] 
= K[t_{1}, \ldots, t_{d}, t_{1}^{-1}, \ldots, t_{d}^{-1}, s]$
denote the Laurent polynomial ring in $2d + 1$ variables over a field $K$. 
If $\alpha = (\alpha_{1}, \ldots, \alpha_{d}) \in \ZZ^{d}$, then
${\bf t}^{\alpha}s$ is the Laurent monomial
$t_{1}^{\alpha_{1}} \cdots t_{d}^{\alpha_{d}}s \in K[{\bf t}, {\bf t}^{-1}, s]$. 
In particular ${\bf t}^{\bf 0}s = s$.
We define the {\em toric ring} of $\Gamma(\Oc(P), - \Cc(Q))$ as the subring 
$K[\Gamma(\Oc(P), - \Cc(Q))]$ of $K[{\bf t}, {\bf t}^{-1}, s]$ which is generated
by those Laurent monomials ${\bf t}^{\alpha}s$ 
with $\alpha \in \Gamma(\Oc(P), - \Cc(Q)) \cap \ZZ^{d}$.
Let 
\[
K[{\bf x}, {\bf y}, z] = K[\{x_{I}\}_{\emptyset \neq I \in \Jc(P)} \cup 
\{y_{\max(J)}\}_{\emptyset \neq J \in \Jc(Q)} \cup \{ z \}]
\]
denote the polynomial ring over $K$
and define the surjective ring homomorphism 
$\pi : K[{\bf x}, {\bf y}, z] \to K[\Gamma(\Oc(P), - \Cc(Q))]$
by the following: 

\begin{itemize}
	\item
	$\pi(x_{I}) = {\bf t}^{\rho(I)}s$, 
	where $\emptyset \neq I \in \Jc(P)$;
	\item
	$\pi(y_{\max(J)}) = {\bf t}^{- \rho(\max(J))}s$,
	where $\emptyset \neq J \in \Jc(Q)$;
	\item
	$\pi(z) = s$.
\end{itemize}
Then the {\em toric ideal} $I_{\Gamma(\Oc(P), - \Cc(Q))}$ of $\Gamma(\Oc(P), - \Cc(Q))$ is 
the kernel of $\pi$. 

Let $<$ denote a reverse lexicographic order on $K[{\bf x}, {\bf y}, z]$
satisfying
\begin{itemize}
	\item
	$z <  y_{\max(J)} <x_{I}$;
	\item
	$x_{I'} < x_{I}$ if $I' \subset I$;
	\item
	$y_{\max(J')} < y_{\max(J)}$ if $J' \subset J$,
\end{itemize}
and $\Gc$ the set of the following binomials:
\begin{enumerate}
	\item[(i)]
	$x_{I}x_{I'} - x_{I\cup I'}x_{I \cap I'}$;
	\item[(ii)]
	$y_{\max(J)}y_{\max(J')} - y_{\max(J\cup J')}y_{\max(J * J')}$;
	\item[(iii)]
	$x_{I}y_{\max(J)} - x_{I \setminus \{p_{i}\}}y_{\max(J) \setminus \{q_{i}\}}$,
\end{enumerate}
where 
\begin{itemize}
	\item
	$x_{\emptyset} = y_{\emptyset} = z$;
	\item
	$I$ and $I'$ are poset ideals of $P$ which are incomparable in $\Jc(P)$;
	\item  
	$J$ and $J'$ are poset ideals of $Q$ which are incomparable in $\Jc(Q)$;
	\item
	$J*J'$ is the poset ideal of $Q$ generated by $\max(J \cap J')\cap (\max(J)\cup \max(J'))$;
	\item $p_{i}$ is a maximal element of $I$ and $q_{i}$ is a maximal element of $J$. 
\end{itemize}

\vspace{2mm}

First, we have 

	\begin{Theorem}
		\label{Boston}
		Work with the same situation as above.
		Then $\Gc$ is a Gr\"obner basis of
		$I_{\Gamma(\Oc(P), - \Cc(Q))}$ with respect to $<$.
	\end{Theorem}
	\begin{proof}
		It is clear that $\Gc \subset I_{\Gamma(\Oc(P), - \Cc(Q))}$.
		For a binomial $f = u - v$, $u$ is called the {\em first}
		monomial of $f$ and $v$ is called the {\em second} monomial of $f$. 
		We note that the initial monomial of each of the binomials (i) -- (iii) 
		with respect to $<$ is its first monomial. 
		Let ${\rm in}_{<}(\Gc)$ denote the set of initial monomials of binomials 
		belonging to $\Gc$.  It follows from \cite[(0.1)]{OHrootsystem} that,
		in order to show that $\Gc$ is a Gr\"obner basis of
		$I_{\Gamma(\Oc(P), - \Cc(Q))}$ with respect to $<$, we must prove 
		the following assertion:
		($\clubsuit$) If $u$ and $v$ are monomials belonging to 
		$K[{\bf x}, {\bf y}, z]$ with $u \neq v$ such that 
		$u \not\in \langle {\rm in}_{<}(\Gc) \rangle$ 
		and $v \not\in \langle {\rm in}_{<}(\Gc) \rangle$,
		then $\pi(u) \neq \pi(v)$.
		
		Let $u, v \in K[{\bf x}, {\bf y}, z]$ be monomials
		with $u \neq v$.  Write
		\[
		u = z^{\alpha} x_{I_{1}}^{\xi_{1}} \cdots x_{I_{a}}^{\xi_{a}}
		y_{\max(J_{1})}^{\nu_{1}} \cdots y_{\max(J_{b})}^{\nu_{b}},
		\, \, \, \, \, \, \, \, \, \, 
		v = z^{\alpha'} x_{I'_{1}}^{\xi'_{1}} \cdots x_{I'_{a'}}^{\xi'_{a'}}
		y_{\max(J'_{1})}^{\nu'_{1}} \cdots y_{\max(J'_{b'})}^{\nu'_{b'}},
		\]
		where
		\begin{itemize}
			\item
			$\alpha \geq 0$, $\alpha' \geq 0$;
			\item
			$I_{1}, \ldots, I_{a}, I'_{1}, \ldots, I'_{a'} 
			\in \Jc(P) \setminus \{ \emptyset \}$;
			\item
			$J_{1}, \ldots, J_{b}, J'_{1}, \ldots, J'_{b'} 
			\in \Jc(Q) \setminus \{ \emptyset \}$;
			\item
			$\xi_{1}, \ldots, \xi_{a}, 
			\nu_{1}, \ldots, \nu_{b},
			\xi'_{1}, \ldots, \xi'_{a'}, 
			\nu'_{1}, \ldots, \nu'_{b'} > 0$,
		\end{itemize}
		and where $u$ and $v$ are relatively prime with
		$u \not\in \langle {\rm in}_{<}(\Gc) \rangle$ 
		and $v \not\in \langle {\rm in}_{<}(\Gc) \rangle$.
		Note that either $\alpha = 0$ or $\alpha' = 0$.
		Hence we may assume that $\alpha' = 0$.  Thus
		\[
		u = z^{\alpha} x_{I_{1}}^{\xi_{1}} \cdots x_{I_{a}}^{\xi_{a}}
		y_{\max(J_{1})}^{\nu_{1}} \cdots y_{\max(J_{b})}^{\nu_{b}},
		\, \, \, \, \, \, \, \, \, \, 
		v = x_{I'_{1}}^{\xi'_{1}} \cdots x_{I'_{a'}}^{\xi'_{a'}}
		y_{\max(J'_{1})}^{\nu'_{1}} \cdots y_{\max(J'_{b'})}^{\nu'_{b'}}.
		\]
		By using (i) and (ii), it follows that
		\begin{itemize}
			\item
			$I_{1} \subsetneq  I_{2} \subsetneq \cdots \subsetneq I_{a}$ and $J_{1} \subsetneq J_{2} \subsetneq \cdots \subsetneq J_{b}$;
			\item
			$I'_{1} \subsetneq I'_{2} \subsetneq \cdots \subsetneq I'_{a'}$ and	$J'_{1} \subsetneq J'_{2} \subsetneq \cdots \subsetneq J'_{b'}$.
		\end{itemize}
		Furthermore, by virtue of \cite{Hibi1987} and \cite{chain}, it suffices to discuss 
		$u$ and $v$ with $(a, a') \neq (0, 0)$ and $(b, b') \neq (0,0)$.
		
		Since $I_a \neq I'_{a'}$, we may assume that $I_a \setminus I'_{a'} \neq \emptyset$.
		Then there exists a maximal element $p_{i^*}$ of $I_a$ with $p_{i^*} \notin I'_{a'}$. 
		
		Now, suppose that $\pi(u)=\pi(v)$.
		Then we have 
		$$\sum\limits_{\stackrel{I \in \{I_1,\ldots,I_a\}}{p_i \in I}}\xi_I-\sum\limits_{\stackrel{J \in \{J_1,\ldots,J_b\}}{q_i \in \max(J)}}\nu_J
		=\sum\limits_{\stackrel{I' \in \{I'_1,\ldots,I'_{a'}\}}{p_i \in I'}}\xi'_{I'}-\sum\limits_{\stackrel{J' \in \{J'_1,\ldots,J'_{b'}\}}{q_i \in \max(J')}}\nu'_{J'}.$$
		for all $1 \leq i \leq d$ by comparing the degree of $t_i$. 
		Since $p_{i^*} \notin I'_{a'}$, one has
		 $$\sum\limits_{\stackrel{I \in \{I_1,\ldots,I_a\}}{p_{i^*} \in I}}\xi_I-\sum\limits_{\stackrel{J \in \{J_1,\ldots,J_b\}}{q_{i^*} \in \max(J)}}\nu_J
		 =-\sum\limits_{\stackrel{J' \in \{J'_1,\ldots,J'_{b'}\}}{q_{i^*} \in \max(J')}}\nu'_{J'}\leq 0.$$
		Moreover,  since $p_{i^*}$ is belonging to $I_a$, we also have 
		$$\sum\limits_{\stackrel{I \in \{I_1,\ldots,I_a\}}{p_{i^*} \in I}}\xi_I>0. $$
		Hence there exists an integer $c$ with $1\leq c \leq b$ such that $q_{i^*}$ is a maximal element of $J_c$.
		Therefore we have $x_{I_a}y_{\max(J_c)} \in \langle{\rm in}_{<}(\Gc) \rangle$, but this is a contradiction.
\end{proof}

This theorem guarantees that the toric ideal $I_{\Gamma(\Oc(P), - \Cc(Q))}$ 
possesses a squarefree initial ideal 
with respect to a reverse lexicographic order for which the variable corresponding to 
the column vector $[0, \ldots, 0, 1]^{t}$ is smallest. 
Therefore, we have the following corollary by using \cite[Lemma 1.1]{HMOS}. 

\begin{Corollary}
	\label{Berkeley}
	For any partially ordered sets $P$ and $Q$ with $|P|=|Q|=d$, $\Gamma(\Oc(P), - \Cc(Q))$ is a normal Gorenstein Fano polytope. 
\end{Corollary}

\section{Example and remark}

As the end of this paper, we give some examples. 
 
Let $\Pc \subset \RR^{d}$ be an integral convex polytope of dimension $d$. 
Given integers $t = 1, 2, \ldots, $ let $i(\Pc, t) := \#(t \Pc \cap \ZZ^{d})$. 
It is known that $i(\Pc, t)$ is a polynomial in $t$ of degree $d$. 
This polynomial is called the {\em Ehrhart polynomial} of $\Pc$. 
The generating function of $i(\Pc, t)$ satisfies
\[
1 + \sum_{t = 1}^{\infty} i(\Pc, t) \lambda^{t} = \frac{\delta_{\Pc}(\lambda)}{(1 - \lambda)^{d + 1}}
\]
where $\delta_{\Pc}(\lambda) = \sum_{i = 0}^{d} \delta_{i}\lambda^{i}$ is a polynomial in $t$ of degree $\le d$. 
Then the vector $(\delta_{0}, \ldots, \delta_{d})$ is called the {\em $\delta$-vector} of $\Pc$. 
It is known that a Fano polytope $\Pc \subset \RR^{d}$ is Gorenstein 
if and only if $\delta_{i} = \delta_{d - i}$ for all $i = 0, 1, \ldots, d$ (See \cite{HibiRedBook}). 

\begin{Example}
Let $P = \{p_1 < p_2 < \cdots < p_d\}$ and $Q = \{q_{i_1} < q_{i_2} < \cdots < q_{i_d}\}$ be chains $(1 \le i_1, \ldots, i_d \le d)$
and $(\delta_{0}, \ldots, \delta_{d})$ be the $\delta$-vector of $\Gamma(\Oc(P), - \Cc(Q))$. 
By \cite[Corollary 1.5]{HMT}, we have $i(\Gamma(\Oc(P), - \Cc(Q)), t) = i(\Gamma(\Cc(Q), - \Cc(P)), t)$, where
$\Gamma(\Cc(Q), - \Cc(P))$ is the convex hull of $\Cc(Q) \cup (- \Cc(P))$. 
Since $\Gamma(\Cc(Q), - \Cc(P))$ is the same as the convex hull of the 
centrally symmetric configuration of $d \times d$ identity matrix, 
hence $\delta_{\Pc}(\lambda) = (1 + \lambda)^d$.  
In particular, $\delta_{i} = \binom{d}{i}$ for all $i = 0, 1, \ldots, d$. 

On the other hand, the convex polytope $\Gamma(\Oc(P), - \Oc(Q))$ which is the convex hull of $\Oc(P) \cup (- \Oc(Q))$ 
is not Gorenstein Fano \cite[Lemma 1.1]{twin}. 
Indeed, the $\delta$-vector of $\Gamma(\Oc(P), - \Oc(Q))$ is $(1, d, 0, \ldots, 0)$. 
\end{Example}
 
\begin{Example}
Let $P = \{ p_1 < p_3, p_2 < p_4 \}$ be a partially ordered set 
and $Q = \{ q_1, q_2, q_3, q_4 \}$ be a partially ordered set such that 
the shape of its Hasse diagram is the same as that of $P$. 
If $Q = \{ q_1 < q_3, q_2 < q_4 \}$ then the $\delta$-vector of $\Gamma(\Oc(P), - \Cc(Q))$ is $(1, 12, 28, 12, 1)$. 
On the other hand, if $Q = \{ q_1 < q_2, q_3 < q_4 \}$ or $\{ q_1 < q_4, q_2 < q_3 \}$, then 
the $\delta$-vector of $\Gamma(\Oc(P), - \Cc(Q))$ is $(1, 12, 26, 12, 1)$. 
\end{Example}

\begin{Remark} 
We proved that the convex polytope $\Gamma(\Cc(P), - \Cc(Q))$ is normal Gorenstein Fano for all partially ordered sets $P, Q$ with $|P| = |Q|$
 $($\cite[Corollary 1.3]{HMT}$)$. 
Moreover, we proved that the Ehrhart polynomial of $\Gamma(\Oc(P), - \Cc(Q))$ is the same as that of $\Gamma(\Cc(P), - \Cc(Q))$ 
for all partially ordered sets $P, Q$ with $|P| = |Q|$. 
In addition, if $P$ and $Q$ possess a common linear extension,  these polytopes $\Gamma(\Oc(P), - \Oc(Q))$, $\Gamma(\Oc(P), - \Cc(Q))$ and 
$\Gamma(\Cc(P), - \Cc(Q))$ have the same Ehrhart polynomial $($\cite[Theorem 1.1]{HMT}$)$. 
\end{Remark}

One of the future problem is to determine the $\delta$-vectors of the above polytopes in terms of partially ordered sets $P$ and $Q$.

\end{document}